\documentclass[11pt]{article}

\usepackage{amsmath,amsfonts,amsthm,amssymb} 
\usepackage{tikz}
\usepackage{lipsum} 
\usetikzlibrary{snakes}
\usepackage{graphicx}
\usepackage{xfrac}   
\usetikzlibrary{calc,3d}
\usetikzlibrary{arrows,backgrounds,snakes,shapes}	

\usepackage{hyperref}
\hypersetup{colorlinks=true,pdftitle="",pdftex}

\newtheorem{thm}{Theorem}[section]
\newtheorem{prop}[thm]{Proposition}
\newtheorem{cor}[thm]{Corollary}
\newtheorem{lem}[thm]{Lemma}

\theoremstyle{definition}
\newtheorem{defn}[thm]{Definition}

\theoremstyle{remark}
\newtheorem{remark}[thm]{Remark}

\def\be{\begin{eqnarray}}
\def\ee{\end{eqnarray}}
\def\ben{\begin{eqnarray*}}
\def\een{\end{eqnarray*}}

\numberwithin{equation}{section}

\newcommand{\zpz}{{\mathbb{Z}}\big/{p\mathbb{Z}}}

\newcommand{\half}{\frac{1}{2}}

\newcommand{\calL}{\mathcal{L}}
\newcommand{\eps}{\varepsilon}

\begin{document}

\title{Entropy Inequalities for Sums in Prime Cyclic Groups\footnote{Some portions of this paper were presented by the authors at the IEEE International Symposia on Information Theory 
in 2014 and 2015 \cite{WWM14:isit, WM15:isit}.}}

\author{
Mokshay Madiman\thanks{Department of Mathematical Sciences, University of Delaware, Newark DE 19716. Email: madiman@udel.edu}, 
Liyao Wang\thanks{J. P. Morgan Chase, New York. Email: njuwangliyao@gmail.com} 
and Jae Oh Woo\thanks{Department of Mathematics, and Department of Electrical and Computer Engineering, The University of Texas at Austin. Email: jaeoh.woo@aya.yale.edu}}

%
%





\date{\today}

\maketitle
\begin{abstract}
Lower bounds for the R\'enyi entropies of sums of independent random variables taking values in cyclic groups of prime order under permutations are established. The main ingredients of our approach are extended rearrangement inequalities in prime cyclic groups building on Lev (2001), and notions of stochastic ordering. Several applications are developed, including to discrete entropy power inequalities, the Littlewood-Offord problem, and counting solutions of certain linear systems.
\end{abstract}




\section{Introduction}

Let $G$ be a finite abelian group with the operation $+$. Specifically, we are interested in the additive prime 
cyclic group $\zpz $ with $a + b = a + b \, (\text{mod}\, p)$ for $a,b\in\zpz$ where $p$ is a prime number with $p\geq 3$. 
For any finite non-empty subsets $A,B\in \zpz$, we define the sumset
\begin{align*}
	A+ B= \left\{ a+ b : a\in A, b\in B \right\}.
\end{align*}
For any random variables $X$ and $Y$ over $\zpz$, we define the sum of two random variables as usual, so that if
the congruence class $X$ were represented by an integer $\tilde{X}$ in it, and $Y$ were represented by $\tilde{Y}$,
then $X+ Y$ is the congruence class represented by $\tilde{X}+\tilde{Y}$ (mod $p$).
In fact, for ease of notation, henceforth we will work with a particular set of representatives for the congruence classes.
Specifically, we will center $\zpz$ at $0$ by thinking of it as the set
\begin{align*}
	\zpz= \bigg\{ -\frac{p-1}{2},\cdots, 0, \cdots, \frac{p-1}{2} \bigg\}.
\end{align*}

Let $f$ be the probability mass function of a random variable $X$ taking values in $\zpz$. 
We denote by $H(X)$ the Shannon entropy of $X$ (or $f$) defined by
\begin{align*}
H(X) 
= H(f) :=  -\sum_{ \mathbb{Z}/p\mathbb{Z}} f(i)\log f(i).
\end{align*} 
It is standard to abuse notation by writing $H(X)$, though the entropy only depends on the probability mass function $f$.

We are interested in the behavior of the entropy of sums in prime cyclic groups, and especially in lower bounds. 
One can easily prove an upper bound on the entropy:
\begin{align}\label{eqn:upper}
H(X+ Y) \leq \min\{ H(X)+H(Y), \log p \}.
\end{align}
In a general setting, the upper bound is well known by data processing inequality leveraging Jensen's inequality. For more details, see, e.g., \cite{Tao10, ALM17, Cover2006}. Given a finite domain, 
the uniform distribution is the entropy maximizer which implies that the upper bound \eqref{eqn:upper} can be tight. 
The more interesting goal, which is the focus of this paper, is to develop {\it sharp  lower bounds for the entropy of the sum under permutations}. 

\par\vspace{.1in}
\noindent{\bf Motivations.}
A fundamental sumset inequality is the Cauchy-Davenport inequality, which asserts that for subsets $A, B$ of $\zpz$, 
\begin{align*}
|A+ B| \geq \min \{|A|+|B|-1, p\}.
\end{align*}
Arithmetic progressions play a special role in the description of extremizers.
In the integer lattice, a similar sumset inequality, which may be thought of as a discrete analogue of the Brunn-Minkowski inequality, 
is proved by 
\cite{GG01}. Such inequalities, together with their stability versions such as Freiman's theorem,
are foundational results of additive combinatorics \cite{TV06:book}.

One motivation for studying lower bounds on entropy of convolutions (or sums of independent random variables)
is to develop an additive combinatorics for probability distributions on groups, or viewed slightly differently to develop an 
embedding of classical additive combinatorics described using the category of sets into a potentially more flexible
theory described using the richer category of probability measures.
Arising from this motivation, entropic analogues of the Pl\"unnecke-Ruzsa-Freiman sumset theory for abelian groups 
have been explored by \cite{Ruz09:1, Tao10, MMT10:itw, MMT12} for the discrete case,
and by \cite{KM14, MS17, MK18} for the locally compact case.
In particular, when $G$ is a torsion-free abelian group, Tao \cite{Tao10} proved that if $X$ and $X'$ are independent and identically distributed
discrete random variables taking values in $G$, then
\begin{align}\label{eqn:tao}
H(X+ X') - H(X) \geq \frac{1}{2}\log 2 - o(1) ,
\end{align}
as $H(X)\to\infty$. 
This is the natural entropy analogue of an implication of the Cauchy-Davenport inequality for the integers, namely that 
$|A+ A|\geq 2|A| (1-o(1))$ as $|A|\to\infty$, given that entropies behave like logarithms of cardinalities (when we identify 
the uniform measure on a set with the set itself).

A second motivation comes from the fact that the ``entropy power inequality'', which is a sharp lower bound for the entropy of 
a convolution of probability densities on a Euclidean space, is extremely useful in a number of fields, including
Information Theory, Probability, and Convex Geometry (see, e.g., the survey \cite{MMX17:0}), and so it is natural to look for similar phenomena in other groups
that might also be of fundamental importance and have significant applicability. There is strong reason to believe
that sharp lower bounds on entropies of sums in discrete groups would have direct applications to communication and control problems;
indeed, preliminary versions of the results of this paper were already used by Chakravorty and Mahajan \cite{CM17} as a key tool in solving a problem from this area.
Motivated by another communication problem involving polar codes, 
Haghighatshoar et al. \cite{HAT14} obtained a nonasymptotic lower bound on the entropy of sums of integer-valued
random variables; this, however, does not match the sharp asymptotic \eqref{eqn:tao}.
Other partial results for integer-valued random variables include 
\cite{HV03, SDM11:isit, JY10}, as well as our companion paper  \cite{MWW17:2} based on Sperner Theory; 
some results are also available for other discrete groups such as the so-called Mrs. Gerber's lemma \cite{WZ73, SW90a, JA14}
in connection with groups of order that is a power of 2. 

A third motivation comes from the literature on small ball probability inequalities of Littlewood-Offord-type.
An advantage of the general formulation of 
our results in terms of R\'enyi entropies is that
it has a pleasant unifying character, containing not just the Cauchy-Davenport inequality, but also
a generalization of the result of Erd\H{o}s \cite{Erd45},
which in turn sharpened a lemma of Littlewood and Offord \cite{LO43},
on the ``anticoncentration'' of linear combinations of Bernoulli random variables.


A fourth motivation comes from the fact that the classical entropy power inequality is closely tied to
the central limit theorem, and discrete entropy power inequalities are expected to have similar close
connections to developing a better understanding of probabilistic limit theorems in discrete settings. 
For example, there is literature on understanding the behavior of entropy and relative entropy
for limit theorems involving Poisson \cite{KHJ05, Yu09:1, HJK10} and compound Poisson \cite{JKM13, BJKM10} limits
on the integers, as well as for those involving the uniform distribution on compact groups \cite{JS00}.

\par\vspace{.1in}
\noindent{\bf Our approach.} 
As mentioned above, the prototype for an entropy lower bound is the entropy power inequality on a Euclidean space,
which has several equivalent formulations. One type of formulation involves bounding from below the entropy of 
the random variable $X+Y$ by the entropies of $X$ and $Y$ respectively-- observe that \eqref{eqn:tao} is a step towards
a similar inequality on the integers. The other type of formulation involves bounding from below the entropy of $X+Y$
by the entropy of $X'+Y'$, where $X'$ and $Y'$ are ``nicely structured'' random variables that have the same entropies
as $X$ and $Y$ respectively. In the Euclidean setting, ``nicely structured'' can mean either Gaussian densities,
or densities that are spherically symmetric decreasing rearrangements of the original densities \cite{WM14}.
Our approach in this paper is to develop inequalities of the second type for the integers and prime cyclic groups.

For a random variable $X$ in $\zpz$, we consider a permutation map $\sigma:\zpz\to\zpz$ acting on $X$ which preserves each corresponding probability mass of $X$, i.e. $\mathbf{P}_X\left( \sigma^{-1}(A) \right)=\mathbf{P}_X(A)$ for each $A\in 2^{\zpz}$. Then we write $\sigma(X)$, a permuted random variable of $X$ permuted by $\sigma$. 
It is easy to see that any permutation maps applied to a random variable preserve the entropy, i.e. $H(X)=H(\sigma(X))$.

In building entropy inequality of sums under permutations, we would like to identify a minimizing permutation pair $(\sigma_1, \sigma_2)$ such that 
\begin{align*}
    H\left( X+ Y \right) \geq H\left( \sigma_1(X) + \sigma_2(Y) \right),
\end{align*}
where for any $(\sigma_1',\sigma_2')$, $H\left( X+ Y \right)\geq H\left( \sigma_1'(X) + \sigma_2'(Y) \right) \geq H\left( \sigma_1(X) + \sigma_2(Y) \right)$. Then we would like to generalize the entropy inequality for any number of sums.

This is a combinatorial question, but what makes the question somewhat challenging is that the entropy of the sum captures both the algebraic structural similarity of supports of the two random variables and the similarity of the shape of the two distributions. We have to control both similarities to minimize the entropy of the sum. Our main result stated in Theorem \ref{thm:main} below does not tell us the explicit form of the permutation pair $(\sigma_1, \sigma_2)$ in general, but one can identify specific ways of $\sigma_1$ and $\sigma_2$ under some shape regularity conditions. Nevertheless, our main result is the first non-trivial entropy lower-bound of the sum in the prime cyclic group, and provides a constructive solution.

Given that measure permutation maps which can be easily seen to be rearrangements of $\zpz$,
notions of rearrangement are key for our study. The main ingredients of our approach are rearrangement inequalities 
in prime cyclic groups and function ordering in the sense of majorization. Such inequalities were first studied by
 Hardy, Littlewood and P\'olya \cite{HLP88:book} on the integers, and extended by Lev \cite{Lev01} to the prime cyclic groups. 
 Based on the foundation these works provide, we will exploit several notions of rearrangements below.

\par\vspace{.1in}
\noindent{\bf Organization of this note.} 
We start with basic definitions of various notions of rearrangement as well as entropy in Section~\ref{sec:prelim}.
Our main results are described in Section~\ref{sec:main} and include a general inequality for R\'enyi entropies
(of arbitrary order $\alpha\in[0,\infty]$) of convolutions. By specializing to specific values of $\alpha$, one
obtains several interesting applications, with the most interesting being to Littlewood-Offord-type phenomena, 
which are discussed in Section \ref{section:applications}. 
A remark on the distributions that appear in the lower bound is made in Section~\ref{sec:canon}.
The subsequent sections contain the proofs of the main results.

\section{Preliminaries}
\label{sec:prelim}

\subsection{Basic rearrangement notions}

Let us consider a probability mass function $f(x):\zpz \to \mathbb{R}_+$ where $\mathbb{R}_+$ is the set of non-negative real-values. 

\begin{defn}
We define $f^+$ as the rearrangement of $f$ such that for any positive $z\in\zpz$,
\begin{align*}
	f^+(0) \geq f^+(+1) \geq f^+(-1) \geq \cdots &\geq f^+(z) \geq f^+(-z) \geq \cdots \\
	&\geq  f^+\left( \frac{p-1}{2} \right) \geq f^+\left( \frac{-p+1}{2} \right).
\end{align*}
Similarly, we define $f^-$ as the rearrangement of $f$ such that for any positive $z\in\zpz$,
\begin{align*}
	f^-(0) \geq f^-(-1) \geq f^-(+1) \geq \cdots &\geq f^-(-z) \geq f^-(+z) \geq \cdots\\
	&\geq  f^-\left( \frac{-p+1}{2} \right) \geq f^-\left( \frac{p-1}{2} \right).
\end{align*}
\end{defn}

By the construction, values of $f^+$ and $f^-$ should have one-to-one correspondence with values of $f$. 
We also note that rearrangements $f^+$ and $f^-$ make the probability mass be almost symmetric and unimodal, 
and $f^+$ and $f^-$ are mirror-images each other with respect to the center at $0\in \zpz$. 
Figure \ref{fig:rearrangement_types} shows an example of rearrangements in $\mathbb{Z}\big/ 17\mathbb{Z}$ for illustration.

{\begin{figure}
	\begin{center}
		\begin{tikzpicture}[scale=0.8]
		\draw (-6,0) -- (0,0);
		\draw[snake=ticks,segment length=0.5cm] (-5.5,0) -- (-0.5,0);
		\node[circle,draw=black,scale=0.3] at (-5.5,1.8) {$$};	
		\node[circle,fill=black,scale=0.3] at (-4.0,2.7) {$$};
		\node[rectangle,draw=black,scale=0.4] at (-1.5,0.3) {$$};
		\node[rectangle,fill=black,scale=0.4] at (-0.5,1.1) {$$};
		
		\draw (-5.5,0) -- (-5.5,1.7);
		\draw (-4.0,0) -- (-4.0,2.6);
		\draw (-1.5,0) -- (-1.5,0.2);
		\draw (-0.5,0) -- (-0.5,1.0);
		
		\node[regular polygon,regular polygon sides=5,fill=black,scale=0.4] at (-5.0,0) {$$};
		\node[regular polygon,regular polygon sides=5,fill=black,scale=0.4] at (-4.5,0) {$$};
		\node[regular polygon,regular polygon sides=5,fill=black,scale=0.4] at (-3.5,0) {$$};
		\node[regular polygon,regular polygon sides=5,fill=black,scale=0.4] at (-3.0,0) {$$};
		\node[regular polygon,regular polygon sides=5,fill=black,scale=0.4] at (-2.5,0) {$$};
		\node[regular polygon,regular polygon sides=5,fill=black,scale=0.4] at (-2.0,0) {$$};
		\node[regular polygon,regular polygon sides=5,fill=black,scale=0.4] at (-1.0,0) {$$};
		
		\node at (-5.5,-0.3) {\small$-7$};
		\node at (-4.0,-0.3) {\small$-4$};
		\node at (-1.5,-0.3) {\small$1$};
		\node at (-0.5,-0.3) {\small$3$};
		\node at (-6.5, 2.5) {{$f$}};
		
		\draw (2,0) -- (8,0);
		\draw[snake=ticks,segment length=0.5cm] (2.5,0) -- (7.5,0);
		\node[rectangle,fill=black,scale=0.4] at (4.5,1.1) {$$};
		\node[circle,fill=black,scale=0.3] at (5.0,2.7) {$$};
		\node[circle,draw=black,scale=0.3] at (5.5,1.8) {$$};
		\node[rectangle,draw=black,scale=0.4] at (6.0,0.3) {$$};
		
		\draw (5.5,0) -- (5.5,1.7);
		\draw (5.0,0) -- (5.0,2.6);
		\draw (6.0,0) -- (6.0,0.2);
		\draw (4.5,0) -- (4.5,1.0);
		
		\node[regular polygon,regular polygon sides=5,fill=black,scale=0.4] at (2.5,0) {$$};
		\node[regular polygon,regular polygon sides=5,fill=black,scale=0.4] at (3.0,0) {$$};
		\node[regular polygon,regular polygon sides=5,fill=black,scale=0.4] at (3.5,0) {$$};
		\node[regular polygon,regular polygon sides=5,fill=black,scale=0.4] at (4.0,0) {$$};
		\node[regular polygon,regular polygon sides=5,fill=black,scale=0.4] at (6.5,0) {$$};
		\node[regular polygon,regular polygon sides=5,fill=black,scale=0.4] at (7.0,0) {$$};
		\node[regular polygon,regular polygon sides=5,fill=black,scale=0.4] at (7.5,0) {$$};
		
		\node at (4.5,-0.3) {\small$-1$};
		\node at (5.0,-0.3) {\small$0$};
		\node at (5.5,-0.3) {\small$1$};
		\node at (6.0,-0.3) {\small$2$};
		\node at (7.5, 2.5) {{$f^{+}$}};
		
		\node at (1,1.5) {$\Rightarrow$};
		
		\draw (2,-6) -- (8,-6);
		\draw[snake=ticks,segment length=0.5cm] (2.5,-6) -- (7.5,-6);
		\node[rectangle,fill=black,scale=0.4] at (5.5,-4.9) {$$};
		\node[circle,fill=black,scale=0.3] at (5.0,-3.3) {$$};
		\node[circle,draw=black,scale=0.3] at (4.5,-4.2) {$$};
		\node[rectangle,draw=black,scale=0.4] at (4.0,-5.7) {$$};
		
		\draw (5.5,-6) -- (5.5,-5.0);
		\draw (5.0,-6) -- (5.0,-3.4);
		\draw (4.5,-6) -- (4.5,-4.3);
		\draw (4.0,-6) -- (4.0,-5.8);
		
		\node[regular polygon,regular polygon sides=5,fill=black,scale=0.4] at (2.5,-6) {$$};
		\node[regular polygon,regular polygon sides=5,fill=black,scale=0.4] at (3.0,-6) {$$};
		\node[regular polygon,regular polygon sides=5,fill=black,scale=0.4] at (3.5,-6) {$$};
		\node[regular polygon,regular polygon sides=5,fill=black,scale=0.4] at (6.0,-6) {$$};
		\node[regular polygon,regular polygon sides=5,fill=black,scale=0.4] at (6.5,-6) {$$};
		\node[regular polygon,regular polygon sides=5,fill=black,scale=0.4] at (7.0,-6) {$$};
		\node[regular polygon,regular polygon sides=5,fill=black,scale=0.4] at (7.5,-6) {$$};
		
		\node at (4.0,-6.3) {\small$-2$};
		\node at (4.5,-6.3) {\small$-1$};
		\node at (5.0,-6.3) {\small$0$};
		\node at (5.5,-6.3) {\small$1$};
		\node at (7.5,-3.5) {{$f^-$}};
		
		\node at (1,-4.5) {$\Rightarrow$};
	\end{tikzpicture}
	\end{center}
	\caption{Types of Rearrangements in $\mathbb{Z}\big/ 17\mathbb{Z}$} \label{fig:rearrangement_types}
\end{figure}
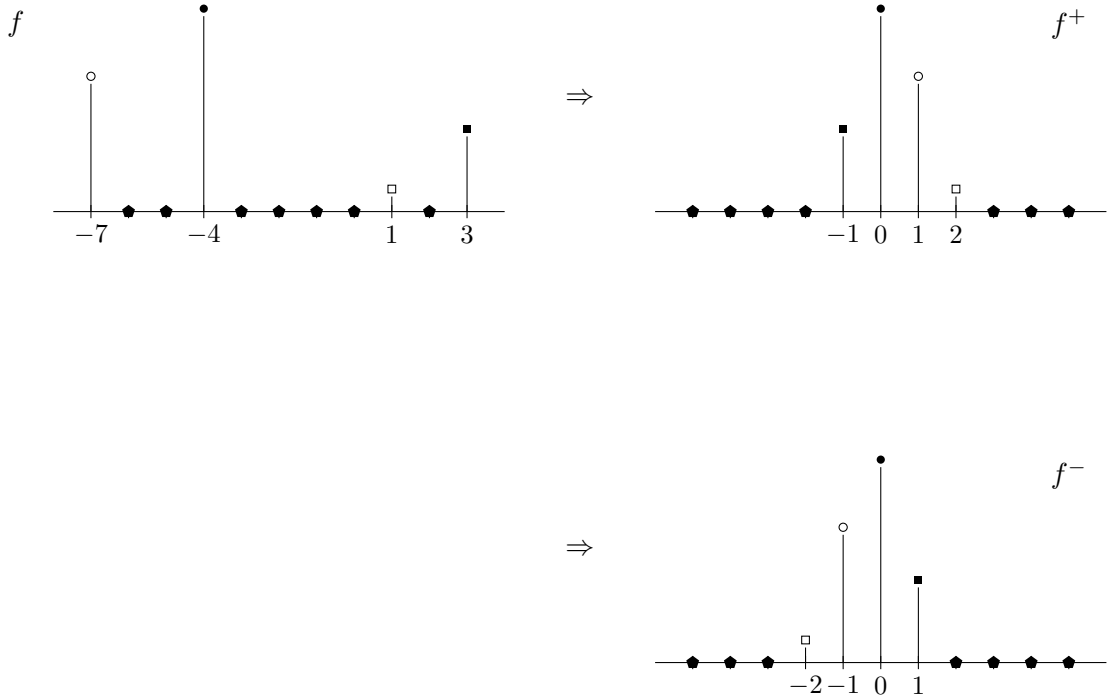}
		
Based on those rearrangements, we consider some regularity assumptions of probability mass functions. 

\begin{defn}
We say that $f$ is {\it $\triangle$-regular} if $f^+(z)=f^-(z)$ for any $z\in\zpz$. In this special case, we sometimes denote $f^+$ or $f^-$ by $f^*$ since they are the same. Equivalently, for any positive $z\in\zpz$,
\begin{align*}
f^*(0) \geq f^*(+1) = f^*(-1) \geq \cdots &\geq f^*(z) = f^*(-z) \geq \cdots\\
&\geq  f^*\left( \frac{p-1}{2} \right) = f^*\left( \frac{-p+1}{2} \right).
\end{align*}
\end{defn}

Clearly, if $f$ is $\triangle$-regular, then $f^*$ is completely symmetric and unimodal. Hence $f$ contains pairs of the same values, except for the largest value that must be taken an odd number of times, or $f$ is a zero function. 
We remark that the definition of $\triangle$-regular is consistent with the 
$*$-symmetrically decreasing definitions of Hardy, Littlewood, P\'olya \cite{HLP88:book} and the balanced function definition of Lev \cite{Lev01}. 
For our purposes, we find it useful to define an additional notion of regularity.

\begin{defn}
We say that $f$ is {\it $\square$-regular} if 1) $f^+(z + 1)=f^-(z)$ for any $z\in\zpz$ and $\left|\text{supp}f \right|$ is even, or 2) $f\equiv 0$. In other words, for any positive $z\in\zpz$ except $(p-1)/2$,
\begin{align*}
f^+(0) = f^+(+1) \geq \cdots &\geq f^+(-z) = f^+(z+1) \geq \cdots\\
&\geq  f^+\left( \frac{-p+3}{2} \right) = f^+\left( \frac{p-1}{2} \right),
\end{align*}
and $f^+\left( \frac{-p+1}{2} \right)=0$.
\end{defn}

If $f$ is $\square$-regular, then $f^+$ is symmetric and unimodal with respect to $+1/2$, and $f^-$ is symmetric and unimodal with respect to $-1/2$. 
Thus $f$ contains only pairs of the same values, or $f$ is a zero function. 

\subsection{Some preliminaries on entropies}

For a general treatment of entropies, we focus on R\'enyi entropies, which generalize the Shannon entropy to a one-parameter family. 

\begin{defn}
The {\it R\'enyi entropy of order $\alpha\in(0,1)\cup(1,+\infty)$}, of a probability mass function $f$ or a random variable $X$ drawn from it, is defined by
\begin{align*}
H_\alpha  (X)= H_\alpha  (f):=\frac{1}{1-\alpha } \log\left( \sum_{i\in \zpz} f(i)^{\alpha } \right).
\end{align*}
For limiting cases of $\alpha$, define
\begin{align*}
H_0(X) &= \log|\text{supp}(f)|,\\
H_1 (X) &= \sum_{i\in \mathbb{Z}/p\mathbb{Z}} -f(i)\log f(i),\\
H_\infty (X) &= -\log \sup_{i\in \mathbb{Z}/p\mathbb{Z}} f(i).
\end{align*}
\end{defn}

As is usual, we often mildly abuse notation by writing $X$ as the argument though all entropies only depend on $f$.
Note that the definitions of $H_0, H_1, H_\infty$ are consistent with taking limits of $H_\alpha(X)$ in the range $\alpha\in(0,1)\cup(1,+\infty)$, 
giving us a natural definition of R\'enyi entropy of order $\alpha \in [0,\infty]$. In particular, $H_1(\cdot)$ is same as the {\it Shannon entropy} $H(\cdot)$. Moreover, when $\alpha=2$, we have
\begin{align*}
H_2 (X) &= -\log\left(\sum_{i\in \mathbb{Z}/p\mathbb{Z}} f(i)^2 \right) = -\log \mathbf{P}\left\{X'=X''\right\}
\end{align*}
where $X'$ and $X''$ are two i.i.d. copies of $X$. Thus $H_2(X)$ captures the probability that two i.i.d. random variables take the same value, 
because of which it is sometimes called the {\it collision entropy}. 
Similarly, since $H_\infty (X) =\min_i -\log f(i)$ when $X$ has finite support, it is sometimes called the {\it min-entropy}.

\section{Main Results}\label{sec:main}

\begin{defn}
We define a collection of ordered index tuples $\mathfrak{I}(f)$ induced by a function $f$ as follows.
\begin{align*}
    \mathfrak{I}(f)=\left\{\left( i_1,\cdots,i_p\right) \big| f(i_1)\geq f(i_2)\geq \cdots \geq f(i_p) \right\}.
\end{align*}
\end{defn}
Because of the possibility of equal values, we note that there may exist more than one such ordered index tuple in $\mathfrak{I}(f)$ for a given $f$. So $\mathfrak{I}(f)$ may contain multiple tuple elements. 

\begin{defn}\label{def:shape-compatibility}
For an index tuple collection $\mathfrak{I}$, we say that $f$ is $\mathfrak{I}$-compatible if $\mathfrak{I}\cap \mathfrak{I}(f)\neq \emptyset$. Then we say that $f$ and $g$ are {{\it shape-compatible}} if both $f$ and $g$ are $\mathfrak{I}(f)\cap \mathfrak{I}(g)$-compatible. More generally, we say that $f_1,\cdots, f_n$ are {{\it shape-compatible}} if all $f_i$'s are $\mathfrak{I}(f_1)\cap\cdots\cap\mathfrak{I}(f_{n})$-compatible.
\end{defn}

We usually omit the common compatible set $\mathfrak{I}(\neq \emptyset)$ for the brevity when we describe the shape-compatibility unless it shows an ambiguity. Informally, if two functions are shape-compatible, the distributions of their values have perfect similarity of shape (where
we may think of shape in the context of plotting a graph of values along $\zpz$). 
For further convenience, we define two order collections $\mathfrak{I}^+$ and $\mathfrak{I}^-$ by
\begin{align*}
\mathfrak{I}^+ &= \left\{ \left( 0, +1, -1, \cdots, +\frac{p-1}{2},-\frac{p-1}{2} \right)\right\},\\
\mathfrak{I}^- &= \left\{\left( 0, -1, +1, \cdots, -\frac{p-1}{2},+\frac{p-1}{2} \right)\right\}.
\end{align*}
Then $\mathfrak{I}^+\subseteq \mathfrak{I}(f^+)$ and $\mathfrak{I}^-\subseteq \mathfrak{I}(f^-)$. We also write $\mathfrak{I}^*\subseteq \mathfrak{I}(f) $ if $\mathfrak{I}(f)$ contains 
both $\mathfrak{I}^+$ and $\mathfrak{I}^-$ simultaneously, i.e., $\mathfrak{I}^*= \mathfrak{I}^+\cup \mathfrak{I}^-$. So if $f$ is $\triangle$-regular, $\mathfrak{I}^* \subseteq \mathfrak{I}(f^*)$. 
These notions enable us to describe the shape of the distribution intuitively. 

\begin{thm}\label{thm:decomposition}
For any probability mass function $f:\zpz\to\mathbb{R}_+$, choose one of ordered index tuples from $\mathfrak{I}(f)$, say $\mathfrak{I}$. Then there exists unique $f_{\triangle}$ and $f_{\square}$ such that
\begin{align*}
f=f_{\triangle} + f_{\square}
\end{align*}
where $f_{\triangle}$ is $\triangle$-regular and $f_{\square}$ is $\square$-regular, and  $f_{\triangle}$, and $f_{\square}$ are shape-compatible.
\end{thm}
\begin{proof}
See Section \ref{sec:prop}.
\end{proof}

We remark that if $f_{\triangle}$ and $f_{\square}$ are shape-compatible, it is easy to see that
\begin{align*}
f^+ = f^+_{\triangle}+f^+_{\square}, \quad f^- = f^-_{\triangle}+f^-_{\square}.
\end{align*}

Let us also note that the decomposition is no longer unique if we remove the assumption that $f_{\triangle}$ and $f_{\square}$ are  shape-compatible.
To see this, let $f=a_1\{i_1\} + a_2\{i_2,i_3\}$ such that $f(i_1)=a_1, f(i_2)=f(i_3)=a_2$ and $a_1\geq a_2$. Then $f$ itself is $\triangle$-regular. So we can find a decomposition $f=f_{\triangle}$. But we can also choose $f'_{\triangle}=a_1\{i_1\}$ and $f'_{\square} = a_2\{i_2,i_3\}$, where $f=f'_{\triangle} + f'_{\square}$. Therefore the uniqueness cannot be guaranteed without the shape-compatibility assumption.

Based on Theorem~\ref{thm:decomposition}, it is more convenient to state our main theorem by using 
function-convolution forms instead of random variable-addition forms. To do that, we define the convolution $\star$ between $f$ and $g$ by
\begin{align*}
f\star g (k) = \sum_{i+ j = k} f(i)g(j).
\end{align*}
Furthermore, as we indicated in the introduction, it is impossible to choose a permutation $\sigma$ from $+$ and $-$ rearrangements in general. One can find examples satisfying
\begin{align*}
H_{\alpha}\left(f^+\star g^+\right) > H_{\alpha}\left(f^+ \star g^-\right), \quad \text{or}\quad H_{\alpha}\left(f^+\star g^-\right) > H_{\alpha}\left(f^+ \star g^+\right), \quad \text{etc.}
\end{align*}
So our answer admits more complicated forms.

\begin{thm}\label{thm:main}
Let $f_1,f_2,\cdots,f_n$ be probability mass functions on $\zpz$.
Then there exists a function
$$
\delta: \{1,\ldots, n\} \times \{ \triangle,\square\}^n \rightarrow \{+,-\}
$$ 
such that
$$
\mathfrak{I}^+ \subseteq \mathfrak{I}\left( f_{1,\omega_1}^{\delta(1,\omega)} \star f_{2,\omega_2}^{\delta(2,\omega)} \star \cdots \star f_{n,\omega_n}^{\delta(n,\omega)}  \right)
$$ 
for each $\omega=(\omega_1,\cdots, \omega_n)\in \{ \triangle,\square\}^n$, where $f_{i,\triangle}$ and $f_{i,\square}$ denote a decomposition of $f$
in the sense of Theorem~\ref{thm:decomposition}.
Moreover, for any $\alpha\in [0,\infty]$ and any such function $\delta$,
\begin{align}\label{eqn:lowerboundentropy}
H_{\alpha}\left( f_1 \star f_2 \star \cdots \star f_n \right) \geq 
H_{\alpha} \left( \sum_{\mathbf{\omega}=(\omega_1,\cdots,\omega_n)\in\{ \triangle,\square\}^n} f_{1,\omega_1}^{\delta(1,\omega)} \star f_{2,\omega_2}^{\delta(2,\omega)} \star \cdots \star f_{n,\omega_n}^{\delta(n,\omega)}  \right) .
\end{align}
\end{thm}
\begin{proof}
See Section \ref{sec:mainthm}.
\end{proof}

We remark that by controlling $\delta(i,\omega)$'s, we are able to keep all summands in \eqref{eqn:lowerboundentropy} to be $\mathfrak{I}^+$-compatible. Otherwise, the inequality \eqref{eqn:lowerboundentropy} can be false. 
The way to choose each $\delta(i,\omega)$'s is discussed in Section \ref{subsec:controlshape} below.

Although the choice of 
the values $\delta(i,\omega)$ is not unique, the lower bound will be the same for all allowed choices. If all $f_i$'s are either $\triangle$-regular or $\square$-regular, then one can easily check that all values of lower bounds are the same
since any of the lower bounds form valid lower bounds for any of the others
by just applying Theorem~\ref{thm:main}. In fact, however, not only is the value
of the entropy in the lower bound the same for all allowed choices, the extremal distribution is also the same no matter how we choose the $\delta(i,\omega)$'s without any regularity assumptions on $f_i$'s. 

\begin{thm}\label{thm:same}
All valid choices of $\delta(i,\omega)$'s in Theorem \ref{thm:main} give the same distribution in the lower bound.
\end{thm}
\begin{proof}
See Section \ref{sec:canon}.
\end{proof}

In fact, one can identify a canonical description of the distribution that appears in the lower bound, as
explained in Section~\ref{sec:canon} where we prove Proposition~\ref{thm:same}.

Let us consider some special cases of Theorem \ref{thm:main} that are easier to state. To start, note that if we assume
regularity of some of the functions to begin with, we can dispense with needing to use decompositions.

\begin{cor}
Let $f$ be an arbitrary probability mass function. Assume $h_1,h_2,$ $\cdots,h_n$ are $\triangle$-regular and $s_1,s_2,\cdots,$
$s_{2m-1},s_{2m}$ are $\square$-regular.
Then
\begin{align*}
H_{\alpha}\big(f\star h_1\star h_2 \star &\cdots \star h_n \star s_1 \star s_2 \star \cdots\star  s_{2m-1} \star s_{2m} \big) \\
& \geq  H_{\alpha} \big( f^{+/-} \star h_1^* \star h_2^* \star \cdots \star h_n^* \star s_1^+ \star s_2^- \star \cdots \star s_{2m-1}^+ \star s_{2m}^- \big) ,
\end{align*}
where $f^{+/-}$ means that we can choose either $f^+$ or $f^-$.
\end{cor}

Theorem \ref{thm:main} takes a particularly simple form when there are only two summands.

\begin{cor}\label{cor:two}
If $f$ and $g$ are probability mass functions on $\zpz$, then
\begin{align*}
H_{\alpha}\left( f \star g \right) \geq H_{\alpha} \left( f^+ \star g_{\triangle}^- + f^- \star g_{\square}^+ \right)
\end{align*}
where $g=g_{\triangle}+g_{\square}$ in the sense of Theorem~\ref{thm:decomposition}.
\end{cor}
\begin{proof}
By applying Theorem \ref{thm:main} and following the rule in Section \ref{subsec:controlshape},
\begin{align*}
H_{\alpha} \left( f \star g \right) \geq H_{\alpha} \left( f^+_{\triangle} \star g^-_{\triangle} + f^-_{\triangle} \star g^+_{\square} + f^+_{\square} \star g^-_{\triangle} + f^-_{\square} \star g^+_{\square} \right).
\end{align*}
Since $f,f_{\triangle}$, and $f_{\square}$ are all shape-compatible, $f^+ = f^+_{\triangle}+f^+_{\square}$ and $f^- = f^-_{\triangle}+f^-_{\square}$. Therefore the Corollary follows.
\end{proof}

\section{Applications} \label{section:applications}

\subsection{Littlewood-Offord problem}
\label{sec:Rinf}

The original Littlewood-Offord lemma is often stated in the following probabilistic language.
Let $X_i$ be i.i.d. Bernoulli random variables, taking values 0 and 1, each with probability $\half$.
Suppose $\mathbf{a}=(a_1, \ldots, a_n)$, where  $a_i$ are nonzero real numbers. Define the random variable
\be\label{eq:sum}
S_{\mathbf{a}} = \sum_{i=1}^n X_i a_i ,
\ee
and let $Q(\mathbf{a})=\max_{x} P (S_\mathbf{a}=x)$. Then Littlewood and Offord \cite{LO43} proved that
\ben
Q(\mathbf{a}) =O(n^{-\half} \log n) ,
\een
and conjectured that the log factor was unnecessary. Erd\H{o}s \cite{Erd45} not only removed the log factor but also
identified an extremal $\mathbf{a}$; specifically he
showed that if $\mathbf{a}^*=(1,\ldots, 1)$ has each entry equal to 1, 
then 
\ben
Q(\mathbf{a}) \leq Q(\mathbf{a}^*) ,
\een
from which the determination of the asymptotic behavior becomes a simple computation.
This result has many beautiful ramifications, variants and generalizations, some of which have found deep applications 
in areas such as random matrix theory; several surveys are now available \cite{NV13, Kri16, GEZ17}.

%
%
%
%
%
%

More generally, instead of defining $Q(\mathbf{a})$ as above, one can study the  maximal ``small ball probability''
$$
Q_{\calL, n, \eps}(\mathbf{a})= \max_x \mathbf{P} \{S_\mathbf{a}\in [x, x+\eps)\},
$$
where we have made the dependence on the law $\calL$ of the random variables $X_i$,
the number $n$ of summands, and the ``bandwidth'' $\eps>0$, explicit in the notation.
Note that $Q_{\calL, n, \eps}$ may be considered as a map from $(\mathbb{R}\setminus\{0\})^n$ to $[0,1]$.
If $X_1,\cdots,X_n$ are i.i.d. random variables distributed according to $\calL$, and  $S_{\mathbf{a}}$ is defined as in \eqref{eq:sum},
the generalized Littlewood-Offord problem asks for estimates of $Q_{\calL, n, \eps}(\mathbf{a})$.
It was shown by Rudelson and Vershynin \cite{RV08:1} that if $X_i$ has a finite third moment,
then the small ball probability is $O(n^{-1/2})$ by way of the Berry-Ess\'een theorem. 
Moreover, there is now a thriving literature on inverse Littlewood-Offord theorems, which 
seem to have originated in work of Arak \cite{Ara80, Ara81} (expressed using concentration functions)
and has been recently rediscovered by Tao, Nguyen and Vu \cite{TV09:2, NV11} with many subsequent rapid developments
(surveys and discussion of history can be found in \cite{GEZ17, NV13}). The general theme here is that 
if the small ball probability $Q_{\calL, n, 0}$ is ``large'' (in the sense that one has lower bounds
that decay sufficiently slowly with $n$), then the coefficient vector $\mathbf{a}$ has a strong additive structure.

Let us discuss how to apply Theorem \ref{thm:main} to problems of Littlewood-Offord type. 
In our setting, we assume that $\mathbf{a}=(a_i)\in\mathbb{Z}^n\setminus \{\mathbf{0}\}$ so that we keep the domain of $a_iX_i$ to be in $\mathbb{Z}$ itself. When $\epsilon<1$, one can easily see that
\begin{align*}
Q_{\calL, n, \eps}(\mathbf{a}) = e^{-H_{\infty}(S_\mathbf{a})}.
\end{align*}
Therefore, maximizing $p_{\epsilon}(\mathbf{a})$ is equivalent to minimizing $H_{\infty}(S_\mathbf{a})$ when $\epsilon <1$. 
More generally, we can answer the same question of the Littlewood-Offord problem for $\zpz$. 
For $a\in\zpz$ and $x\in\zpz$, 
we are able to give not only an explicit maximizer of the Littlewood-Offord problem in $\zpz$ but a more general answer in the sense of the entropy by applying Theorem \ref{thm:main} even without identically distributed assumption. 

\begin{defn}
For some $c\in\zpz$, we say $g$ is $c$-{\it circular} with respect to $f$ if $f(i+ c)=g(i)$. 
\end{defn}

If $g$ is $c$-circular with respect to $f$, one can easily check that $g\star h$ is $c$-circular with respect to $f\star h$ and 
$H_{\alpha} (g\star h)=H_{\alpha} (f\star h)$. As a special case, if $f$ is $\square$-regular, $f^-$ is $1$-circular with respect to $f^+$. 
Therefore if $f_i$ is either $\triangle$-regular or $\square$-regular, 
$H_{\alpha}\left( f_1^{\delta(1)}\star \cdots \star f_{n}^{\delta(n)} \right)$ is preserved for any choice of $\delta(i)\in\{+,-\}$.

For notational convenience, we write $X^+$ and $X^-$ to mean random variables with the probability mass function $f^+$ and $f^-$ respectively.

\begin{thm}\label{thm:lo}
Let $X_1, \ldots, X_n$ be independent random variables taking values in $\zpz$.
For each $i$, suppose that the probability mass function $f_i$ of $X_i$
is either $\triangle$-regular or $\square$-regular, and $f_i$ itself is the same as either $f^*$, $f^+$, or $f^-$.
If $S_\mathbf{a}=a_1X_1 + \cdots + a_nX_n$ where $\mathbf{a}=(a_i)\in\mathbb{Z}^n\setminus \{\mathbf{0}\}$, then
\begin{align*}
H_{\alpha}\left( S_\mathbf{a} \right) \geq H_{\alpha}\left (X_1+ \cdots + X_n \right).
\end{align*}
\end{thm}

\begin{proof}
We observe that $a_iX_i$ does not change values in the probability mass functions. Therefore $(a_iX_i)^+=X_i^+$ and $(a_iX_i)^-=X_i^-$. Since $f_i$ is either $\triangle$-regular or $\square$-regular, the decomposition in the sense of Theorem~\ref{thm:decomposition} admits only one part. Since the entropy is preserved without rearrangements by the circularity, the conclusion follows.
\end{proof}

Observe that $\alpha=+\infty$ with the identically distributed assumption recovers a theorem of Littlewood-Offord type
in $\zpz$. 

Moreover, combining with a result of Kanter \cite[Lemma 4.3]{Kan76:1} or 
Mattner and Roos \cite[Theorem 2.1]{MR07}, one can obtain an estimate of the small ball probability 
as well as the explicit maximizer in some special cases. We first state the required lemma.

\begin{lem}[Kanter in \cite{Kan76:1}, Mattner and Roos in \cite{MR07}]\label{lem:kanter}
Let $X_1,\cdots,X_n$ be independent random variables such that $\mathbf{P}\left(X_i=0 \right) = q_i$ and 
$\mathbf{P}\left( X_i=-1 \right) = \mathbf{P}\left( X_i=+1 \right)$ $= \frac{1-q_i}{2}$, then
\begin{align*}
\mathbf{P}\left( X_1+\cdots + X_n = 0 \text{ or } 1  \right) \leq G\left(\sum_{i=1}^{n} (1-q_i)\right)
\end{align*}
where $G(x)=e^{-x}\left(I_0(x)+I_1(x) \right)$ and $I_k$ is the modified Bessel function of the order $k$. Furthermore, 
$G(x)$ is a complete monotone function with $G(0)=1$ and $G(x)\to 0$ as $x\to\infty$ such that for $x>0$
\begin{align*}
G(x) < \sqrt{\frac{2}{\pi x}}.
\end{align*}
\end{lem}

Then we are able to find an estimate of the small ball probability in a some non-identically distributed cases.

\begin{cor}
Let $X_1,\cdots,X_n$ be independent random variables such that \\ $\mathbf{P}\left(X_i=0 \right)= q_i$ and 
$\mathbf{P}\left( X_i=-1 \right) = \mathbf{P}\left( X_i=+1 \right) = \frac{1-q_i}{2}$.
Suppose for each $i$, $1-q_i<\xi$ for some fixed $\xi>0$.
If $S_\mathbf{a}=a_1X_1 + \cdots + a_nX_n$ where $\mathbf{a}=(a_i)\in\mathbb{Z}^n\setminus \{\mathbf{0}\}$, then
\begin{align*}
\max_x \mathbf{P} \{S_\mathbf{a}=x\} \leq O\left( n^{-1/2} \right).
\end{align*}
\end{cor}
\begin{proof}
By Theorem~\ref{thm:lo} and Lemma \ref{lem:kanter},
\begin{align*}
Q_{\calL, n, 0}(\mathbf{a})\leq \mathbf{P}\left( X_1+\cdots + X_n =0 \right) < G\left( \sum_{i=1}^{n} (1-q_i)\right) \leq G(\xi n) = O\left( n^{-1/2} \right).
\end{align*}
\end{proof}
We remark that finding the small ball probability is originated from Kolmogorov and Rogozin. See \cite{Ess66} for further discussion and details.

\subsection{Collision entropy and counting solutions of linear equations}
\label{sec:R2}

Let $A_i\subseteq \zpz$ be finite sets in $\zpz$ for $i=1,\cdots,n$. Similar to the Littlewood-Offord problem, let us consider 
a linear equation in $\zpz$ for non-zero $c_i\in\zpz$ as follows:
\begin{align}\label{eq:collision}
c_1 x_1 + \cdots + c_nx_n = c_1 x_1' + \cdots + c_nx_n'
\end{align}
where $x_i, x_i'\in A_i$ are equation variables. Then we are able to discuss the number of solutions of the equation \eqref{eq:collision} by applying Theorem \ref{thm:main}. 
\begin{cor}
Among any choices of non-zero $c_i\in\zpz$, the number of solutions of the equation \eqref{eq:collision} is maximized when $c_1=\cdots=c_n=1$.
\end{cor}
\begin{proof}
For each $i$, let $X_i$ and $X_i'$ be i.i.d. uniform distributions on $A_i$. Similarly, 
let $Y_i$ and $Y_i'$ be i.i.d. uniform distributions on $A_i^{\delta'(A_i)}$ where $A_i^{\delta'(A_i)}$ with $\delta'(A_i)\in\{+,-\}$ such that $\bar{\delta}:=\sum_{i=1}^{n} \bar{\delta}(i,\delta'(A_i)) \in\{0,+1\}$ where $\bar{\delta}(\cdot,\delta'(\cdot))$ is defined in Section \ref{subsec:gen_rearrangement} below.
By applying Theorem \ref{thm:main} when $\alpha=2$, we have
\begin{align*}
&-\log \mathbf{P} \big\{ c_1X_1+\cdots+c_nX_n = c_1X_1'+\cdots+c_nX_n' \big\} \\
\geq &-\log \mathbf{P} \big\{ Y_1+\cdots+ Y_n = Y_1'+\cdots+Y_n' \big\}.
\end{align*}
By removing the $\log$-terms,
\begin{align*}
\mathbf{P} \big\{ c_1X_1+\cdots+c_nX_n = c_1X_1'+\cdots+c_nX_n' \big\} 
\leq \mathbf{P} \big\{ Y_1+\cdots+ Y_n = Y_1'+\cdots+Y_n' \big\}.
\end{align*}
We can reinterpret the above inequality as follows:
\begin{align*}
\frac{\text{$\#$ of solutions in the equation of \eqref{eq:collision}}}{\prod_{i}^{n}|A_i|} \leq \frac{\text{$\#$ of 
solutions in the rearranged equation}}{\prod_{i}^{n}\left| A_i^{\delta(A_i)} \right|}.
\end{align*}
Since the two denominators are the same, the proof is complete.
\end{proof}

\subsection{A discrete entropy power inequality}
\label{sec:R1}

As mentioned earlier, the entropy power inequality for random variables in a Euclidean space, originally due to \cite{Sha48, Sta59} 
and since considerably refined (see, e.g., \cite{MG17} and references therein), is very well known and has been deeply studied. 
If two independent random variables $X$ and $Y$ admit density functions $f(x)$ and $g(x)$ on $\mathbb{R}^d$, the entropy power inequality says that
\begin{align}\label{eqn:cont_epi}
\mathcal{N}(X+Y) \geq \mathcal{N}(X) + \mathcal{N}(Y) ,
\end{align}
where $\mathcal{N}(X)=\frac{1}{2\pi e}e^{\frac{2}{d}h(X)}$ and $h(X)=-\int_{\mathbb{R}^d} f(x)\log f(x) dx$,
with equality if and only if $X$ and $Y$ are Gaussian with proportional covariance matrices. We remark that 
$\mathcal{N}(\cdot)$ is called the {\it entropy power}; the name of the inequality comes from the fact that it 
expresses the superadditivity of the entropy power with respect to convolution.

Theorem \ref{thm:main} (or Corollary \ref{cor:two}) specialized to $\alpha=1$ yields a {\it discrete} entropy power inequality
of similar form for {\it uniform} distributions in $\mathbb{Z}$. 

\begin{thm}\label{thm:disc-unif-epi}
If $X$ and $Y$ are independent and uniformly distributed on finite sets $A$ and $B$ in $\mathbb{Z}$,
\begin{align*}
N(X+Y)+1 \geq N(X)+N(Y)
\end{align*}
where $N(X)=e^{2H(X)}$. The equality holds when one of $X$ and $Y$ is a dirac (one-point) mass.
\end{thm}

The extra additive term 1 is akin to the extra factor of 1 that appears in the Cauchy-Davenport inequality when
compared to its continuous analogue, namely the Brunn-Minkowski inequality on $\mathbb{R}$, which says that the 
Lebesgue measure (length) of $A+B$ is at least the sum of lengths of $A$ and $B$ respectively.
Theorem~\ref{thm:disc-unif-epi} is notable since despite much work attempting to prove an entropy power inequality 
on the integers, all prior results  of similar form held only for small classes of distributions (such as binomials \cite{HV03, SDM11:isit},
supported on a set of contiguous integers),
whereas Theorem~\ref{thm:disc-unif-epi} applies to the infinite-dimensional class of uniform distributions on arbitrary finite sets.
We omit the proof of Theorem~\ref{thm:disc-unif-epi} since it is very similar to the deduction in our previous work \cite[Theorem II.2]{WM15:isit}, where we 
started from a lower bound for the entropy of the sum based on the notion of $\#$-log-concave \cite{WWM14:isit}
rather than from Theorem \ref{thm:main} (which provides an alternate path). 

We remark that the exponent $2$ in the entropy power $N(X)=e^{2H(X)}$ is the reasonable choice based on the asymptotically sharp constant by Tao \cite{Tao10}
mentioned in the introduction. Otherwise, it violates Tao's asymptotically sharp constant (see \cite{WM15:isit} for details). 
We also remark that the discrete entropy power inequality for more general distributions in $(\zpz)^d$ remains an interesting open problem.

An immediate consequence of Theorem~\ref{thm:disc-unif-epi} is an explicit bound on entropy doubling. Indeed, since
$X, X'$ are i.i.d., Theorem~\ref{thm:disc-unif-epi} says that $N(X+X') \geq 2N(X)-1$. Taking logarithms, we may write this as
$$
H(X+X')\geq \frac{1}{2}\log\big[2N(X)\big] +\log \bigg[1-\frac{1}{2N(X)}\bigg]
= H(X) + \frac{1}{2}\log 2 -o(1) , 
$$
as $N(X)\to\infty$, which recovers the asymptotic entropy gap on doubling of  \eqref{eqn:tao} from \cite{Tao10}
with significantly less effort for the special class of uniform distributions. Interestingly, however, this result (although
sharp for general distributions) can be improved for uniform distributions by a slightly more involved argument starting
from our main results. 

\begin{cor}\label{cor:asymp-unif}
If $X$ and $X'$ are independent and both uniformly distributed on a finite subset $A$ of $\mathbb{Z}$, then
\begin{align*}
H(X+X') - H(X) \geq \frac{1}{2}-o(1) ,
\end{align*}
where $o(1)$ vanishes as $H(X)\to+\infty$.
\end{cor}
\begin{proof}
Assume $|A| = n$ and let $f=\frac{1}{n}A$. Then $H(X)=\log n$. As $H(X)\to +\infty$, $n\to+\infty$. Then by choosing a sufficiently large $p\gg n$ and applying Corollary \ref{cor:two} with the direct calculation,
\begin{align*}
H(X + X') - H(X) &\geq H(f^{+} \star f^{-}) - H(f)\\
&= -2 \sum_{i=1}^{n-1} \frac{i}{n^2} \log \frac{i}{n^2}  + \frac{1}{n}\log n - \log n\\
&= \bigg( 1-\frac{1}{n} \bigg) \log n -2 \sum_{i=1}^{n-1} \frac{i}{n^2} \log i \\
&\geq \frac{1}{2} -\frac{\log n}{n} -\frac{1}{2n^2} = \frac{1}{2} - o(1).
\end{align*}
\end{proof}

Whereas the asymptotic entropy gap on doubling for general distributions is at least $\frac{1}{2}\log 2$ from \cite{Tao10}, 
Corollary~\ref{cor:asymp-unif} says that the  asymptotic entropy gap on doubling for uniform distributions
exceeds the larger constant $\frac{1}{2}$.

We note that 
$H(X+X')=H(X)=\log p$ which is the maximum entropy when $X$ is uniform on $\zpz$. Therefore, 
it is not possible to bound $H(X+X')-H(X)$ away from 0 in the case of finite cyclic groups. On the other hand, we also note that if $p$ is sufficiently large i.e. $p\gg n$, the behavior of entropy gap is the same as on the integer domain. In such case, as shown in the proof of Theorem II.2 of \cite{WM15:isit}, $H(X+X')-H(X)\geq \frac{1}{2}\log 2$ for i.i.d. uniform distributions of $X$ and $X'$. More general R\'enyi entropy case can be deduced from the proof of the Lemma 2.5 of \cite{MWW17:2}.

\begin{remark}
In fact, as hinted in the introduction, we can directly view Corollary \ref{cor:two} specialized to $\alpha=1$ 
as a discrete entropy power inequality. To see this, note that the original entropy power inequality \eqref{eqn:cont_epi}
may be restated as saying that, if  $Z_1$ and $Z_2$ are independent Gaussian random vectors with $h(X)=h(Z_1)$ and $h(Y)=h(Z_2)$,
then $h(X_1+X_2)\geq h(Z_1+Z_2)$; there is also a strengthening of this available where spherically symmetric decreasing rearrangements
of the densities are used instead of Gaussians \cite{WM13:isit, WM14}. When $\alpha=1$, Corollary \ref{cor:two} is a statement
very much like this but for the integers (there is the additional subtlety of the decomposition being invoked for one of the distributions,
but this is unavoidable as discussed in Section~\ref{sec:main}), and hence can in itself be considered as a discrete entropy power inequality.
\end{remark}

\begin{remark}
The cardinality-minimizing sumset in the integer lattice, explored by Gardner and Gronchi \cite{GG01}, is related to the
so-called $B$-order. We note that this is equivalent to our entropy minimizer in the integers when we restrict to uniform
distributions. However, we also observe that following the $B$-order cannot give an entropy minimizer of the sum 
for uniform distributions in the integer lattice in general.
\end{remark}

\subsection{The Cauchy-Davenport theorem}
\label{sec:R0}

Observe that Corollary \ref{cor:two} contains the Cauchy-Davenport theorem.

\begin{cor}[Cauchy-Davenport Theorem]
If $A$ and $B$ are non-empty finite subsets in $\zpz$, then
\begin{align*}
|A + B| \geq \min \{ |A|+|B|-1,p \}.
\end{align*}
\end{cor}
\begin{proof}
Choose $f=\frac{1}{|A|}1_A$ and $g=\frac{1}{|B|}1_B$ to be the uniform distributions on $A$ and $B$. Then
\begin{align*}
\text{supp}(f \star g) = A + B, \quad \text{supp}(f^+ \star g^-) = A^+ + B^-
\end{align*}
where $A^+:=\text{supp}(f^+)$ and $B^-=\text{supp}(g^-)$. Then $|A^+|=|A|$ and $|B^-|=|B|$. If $|B|$ is odd, $g$ is $\triangle$-regular (so $g^+=g^-$). If $|B|$ is even and $|A|$ is odd, we can change the role of $A$ and $B$. If both $|A|$ and $|B|$ are even, $f$ and $g$ are $\Box$-regular. In any case, Corollary \ref{cor:two} for $\alpha=0$ implies
\begin{align*}
H_0\left( f\star g \right) \geq H_0\left( f^+ \star g^- \right).
\end{align*}
Equivalently,
\begin{align*}
|A+ B| \geq |A^+ + B^-|.
\end{align*}
Since $A^+$ and $B^-$ are sets of consecutive values, $A^+ + B^-$ is again a set of consecutive values (in the sense of cyclic values) in $\zpz$. So $|A^+ + B^-|=|A^+|+|B^-|-1 = |A|+|B|-1$ in $\zpz$ when $|A|+|B|\leq p+1$. Otherwise, $|A^+ + B^-|=p$ since the domain is already full. Therefore, $|A^+ + B^-|=\min \{|A|+|B|-1,p\}$ in $\zpz$ recovers the Cauchy-Davenport theorem.
\end{proof}

We note that this last implication is more or less redundant, since our main theorem is heavily based on Lev's 
set rearrangement majorization (Lemma \ref{lem:lev_set} below), which in turn is originally based on the generalized 
Cauchy-Davenport theorem proved by Pollard \cite{Pol75}; nonetheless it is included for completeness. 

In the Euclidean setting of $\mathbb{R}^d$, the general R\'enyi entropy power inequalities of \cite{WM14} specialize
both to the Shannon-Stam entropy power inequality as well as to the Brunn-Minkowski inequality. Our development above,
where we use a general R\'enyi entropy power inequality on $\zpz$ to obtain both a discrete entropy power inequality
and the Cauchy-Davenport theorem, is exactly the discrete analogue of this phenomenon but is limited to one dimension.
Given that there are some multidimensional versions of sumset inequalities known (see, e.g., \cite{GG01}), it is tempting to
hope that the results of this note can be generalized in an appropriate way to distributions on $(\zpz)^d$. However,
one should temper this hope with the observation that the analogy between sumset inequalities and entropy inequalities
in the Euclidean setting breaks down in some unexpected ways when $d>1$, as discussed extensively in \cite{FMMZ16, FMMZ17}.

\section{Proof of Theorem~\ref{thm:decomposition}}\label{sec:prop}

Given a fixed $\mathfrak{I}=\left\{\left( i_1,\cdots, i_p \right)\right\}\subseteq \mathfrak{I}(f)$ satisfying $f(i_1)\geq f(i_2)\geq\cdots\geq f(i_p)$, we write $f^{[r]} = f(i_r)$ for abbreviation. i.e. $f^{[r]}$ indicates the $r$-th largest value of $f$ and it is achieved at $i_r\in\zpz$.

Firstly, we show the uniqueness. 
Suppose that there exist $f_\triangle, f_\square, f'_\triangle$, and $f'_{\square}$ such that all are $\mathfrak{I}$-compatible and
\begin{align*}
f=f_\triangle+f_{\square}=f'_{\triangle}+f'_{\square}.
\end{align*}
This implies $f_{\triangle}-f'_{\triangle}=f'_{\square}-f_{\square}$. By the shape-compatibility, for any $i_r$ from the tuple in $\mathfrak{I}$
\begin{align*}
\left(f_{\triangle}-f'_{\triangle} \right)(i_r)&=f^{[r]}_{\triangle}-f'^{[r]}_{\triangle},\\
\left( f_{\square}-f'_{\square} \right)(i_r)&=f^{[r]}_{\square}-f'^{[r]}_{\square}.
\end{align*}  
By the construction, $|\text{supp}(f'_{\square}-f_{\square})|$ should be an even number or zero. If $|\text{supp}(f'_{\square}-f_{\square})|$ is zero, we are done since $f'_{\square}-f_{\square}=0$. Suppose that $|\text{supp}(f'_{\square}-f_{\square})|$ is non-zero and even. Since there exist same values in pairs for $f_{\triangle}$ and $f'_{\triangle}$ except the largest one, $\left(f_{\triangle}-f'_{\triangle} \right)(i_1)=f^{[1]}_{\triangle}-f'^{[1]}_{\triangle}=0$ to have the even support size. This again implies $\left( f_{\square}-f'_{\square} \right)(i_1)=f^{[1]}_{\square}-f'^{[1]}_{\square}=0$. Then since there exist same values in pairs for $f_{\square}$ and $f'_{\square}$, $f^{[2]}_{\square}-f'^{[2]}_{\square}=0$. By repeating this process iteratively, we have $f^{[r]}_{\triangle}-f'^{[r]}_{\triangle}=0$ and $f^{[r]}_{\square}-f'^{[r]}_{\square}=0$ for all $r=1,\cdots,p$. Then it contradicts to the assumption that $|\text{supp}(f'_{\square}-f_{\square})|$ is non-zero. 

Secondly, we show the decomposition. To describe an indicator (or characteristic) function of a set conveniently, we mean that $f=aA$ is equivalent to $f(i)=a$ at $i\in A\subseteq \zpz$ for some $a\in\mathbb{R}$, and $f(i)=0$ otherwise. Let $a_r = f^{[r]}-f^{[r+1]}$ for $r=1,\cdots,p-1$ and $a_p=f^{[p]}$. Choose
\begin{align*}
f_{\triangle} &= a_1 \{i_1\} + a_3 \{i_1, i_2, i_3\} + \cdots + a_{p}\{i_1,\cdots, i_p \}, \\
f_{\square} &= a_2 \{i_1,i_2\} + a_4 \{i_1, i_2, i_3 ,i_4 \} + \cdots + a_{p-1} \{i_1,\cdots, i_{p-1}\}.
\end{align*}
Then we can easily check that $f_{\triangle}$ and $f_{\square}$ satisfy the desired properties.

\section{Proof of Theorem \ref{thm:main}}\label{sec:mainthm}

Proving Theorem \ref{thm:main} requires several different results seemingly unrelevant. The structure of the proof is the following. Firstly, we introduce the notion of majorization as a function ordering and some useful results. Secondly, we generalize the Lev-Hardy-Littlewood-P\'olya's rearrangement inequality. Thirdly, we discuss the shape of the rearrangement convolutions. Fourthly, we finish the proof by connecting the rearrangement inequality and the majorization. Finally, we explain how to control rearrangements in the statement of Theorem \ref{thm:main}.

\subsection{Function ordering by majorization}
Consider two real-valued functions $f$ and $g$ in $\zpz$. We write $f \prec g$ if for all $r=1,\cdots,p-1$
\begin{align*}
\sum_{i=1}^r f^{[i]} \leq \sum_{i=1}^{r}g^{[i]},
\end{align*}
and
\begin{align}\label{eqn:maj_equality}
\sum_{i=1}^{p} f^{[i]} = \sum_{i=1}^{p} g^{[i]}.
\end{align}
In this case, we call that $f$ is majorized by $g$. 
It is classical that the majorization of finitely many supported functions imply the inequality of the convex functional sums \cite{HLP88:book, Arn12:book, MOA11:book}. 
\begin{lem}\cite[Proposition 3-C.1]{MOA11:book}\label{lem:maj_to_convex}
Assume that $f$ and $g$ are non-negative functions in $\zpz$ and $f\prec g$. For any convex function $\Phi:\mathbb{R}\to\mathbb{R}$,
\begin{align*}
\sum_{i\in\zpz}\Phi\left( f(i) \right) \leq \sum_{i\in\zpz}\Phi\left( g(i) \right).
\end{align*}
\end{lem}

\subsection{Generalized rearrangement inequality in $\zpz$}\label{subsec:gen_rearrangement}
Lev \cite{Lev01} extended the Hardy-Littlewood-P\'olya's rearrangement inequality \cite{HLP88:book} based on Pollard's generalized Cauchy-Davenport theorem \cite{Pol75}.
\begin{thm}[Lev-Hardy-Littlewood-P\'olya's rearrangement inequality]
Let $f$ and $g$ be arbitrary nonnegative functions in $\zpz$. Assume $h_1,h_2,\cdots,h_n$ are $\triangle$-regular in $\zpz$,
\begin{align*}
f\star g\star h_1\star h_2 \star \cdots \star h_n (0) \leq f^+ \star g^-\star h_1^* \star h_2^* \star \cdots \star h_n^* (0).
\end{align*}
\end{thm}
As explained in the second part of Section \ref{sec:prop}, we use the set-indicator function ambiguation again. In other words, we admit that a set is equivalently an indicator function. For any finite set $A\subseteq\zpz$, $A(i)=1$ if $i\in A$ and $A(i)=0$ if $i\notin A$. Let $\delta'(A)\in \{+,-\}$. We also define a function $\bar{\delta}(A,\delta'(A))$ by
\begin{displaymath}
\bar{\delta}(A,\delta'(A)) =  \left \{ 
\begin{array}{ll}
-1 & \text{if $|A|$ is even and $\delta'(A)=-$,}\\
0 & \text{if $|A|$ is odd,}\\
+1 & \text{if $|A|$ is even and $\delta'(A)=+$.}
\end{array}
\right.
\end{displaymath}
We note that $\bar{\delta}(\cdot,\cdot)$ captures the balanced amount from the center with respect to the convolution. Then the following lemma is proved by Lev as well \cite[Theorem 1]{Lev01}.
\begin{lem}[Lev's set rearrangement majorization]\label{lem:lev_set}
If $A_1,A_2,\cdots,A_n$ are finite sets (or indicator functions) in $\zpz$,
\begin{align*}
{A_1} \star A_2 \star \cdots  \star A_n \prec A_1^{\delta'(A_1)} \star A_2^{\delta'(A_2)} \star \cdots  \star A_n^{\delta'(A_n)}
\end{align*}
where $\delta'(A_i)\in\{+,-\}$ such that $\bar{\delta}:=\bar{\delta}(A_1, \delta'(A_1))+\bar{\delta}(A_2, \delta'(A_2))+\cdots + \bar{\delta}(A_n, \delta'(A_n))\in \{-1,0,1\}$.

Furthermore,
\begin{displaymath}
\mathfrak{I}\left( A_1^{\delta'(A_1)} \star A_2^{\delta'(A_2)} \star \cdots  \star A_n^{\delta'(A_n)} \right) \supseteq \left \{
\begin{array}{ll}
\mathfrak{I}^- & \text{if $\bar{\delta}=-1$}\\
\mathfrak{I}^* & \text{if $\bar{\delta}=0$}\\
\mathfrak{I}^+ & \text{if $\bar{\delta}=+1$}.
\end{array}
\right.
\end{displaymath}
\end{lem}
We note that $A_1+ A_2$ and $A_1 \star A_2$ result in different (set) functions in general. $\text{supp}(A_1\star A_2)=A_1+ A_2$ and $A_1 \star A_2$ counts the multiplicity of the sum in a way of the convolution.

\begin{thm}[Generalized rearrangement inequality]\label{thm:general_rearrangement_inequality}
Let $f$ and $g$ be arbitrary nonnegative functions in $\zpz$. Assume $h_1,h_2,\cdots,h_n$ are $\triangle$-regular and $s_1,s_2,\cdots, $ $s_{2m-1},s_{2m}$ are $\square$-regular in $\zpz$,
\begin{align*}
f\star g\star &h_1\star h_2 \star \cdots \star h_n \star s_1 \star s_2 \star \cdots\star  s_{2m-1} \star s_{2m} (0) \\
&\leq f^+ \star g^-\star h_1^* \star h_2^* \star \cdots \star h_n^* \star s_1^+ \star s_2^- \star \cdots \star s_{2m-1}^+ \star s_{2m}^- (0).
\end{align*}
\end{thm}
\begin{proof}[Sketch of the proof]
The proof idea is the same as Hardy-Littlewood-P\'olya's one, or Lev's one. See \cite[Theorem 373]{HLP88:book} or \cite[Theorem 2]{Lev01} for details. Similar to the decomposition explained in the second part of the Section \ref{sec:prop}, we apply the set-function decomposition of each function. i.e. $f=\sum a_r A_r$ where $a_r=f^{[r]}-f^{[r+1]}$ for $r=1,\cdots,p-1$ and $a_p=f^{[p]}$, $A_r=\{i_1,\cdots,i_r\}$ for $\left( i_1,\cdots,i_p\right)\in \mathfrak{I}(f)$. Then the rearrangement inequality can be reduced to prove the set rearrangement inequality. The key step for the reduction is using the following facts:
\begin{itemize}
\item{For non-zero $\triangle$-regular $h$, its set-function decomposition always involves odd sized sets,}
\item{For non-zero $\square$-regular $s$, its set-function decomposition always involves even sized sets.}
\end{itemize}
Using the reduction, the proof can be finished by recovering the rearrangement inequality with summing up the reduced set rearrangement inequalities. Then it remains to prove the following the set rearrangement inequality which is a special case of Lemma \ref{lem:lev_set}.
\begin{lem}
Assume that $A_f$ and $B_g$ are arbitrary finite sets, $C_{h_1},C_{h_2},\cdots,C_{h_n}$ are odd-sized sets and $D_{s_1},D_{s_2},\cdots,D_{s_{2m-1}},D_{s_{2m}}$ are even-sized sets. Then
\begin{align*}
A_f\star B_g\star &C_{h_1}\star C_{h_2} \star \cdots \star C_{h_n} \star D_{s_1} \star D_{s_2} \star \cdots\star D_{s_{2m-1}} \star D_{s_{2m}} (0) \\
&\leq A_{f}^+ \star B_{g}^-\star C_{h_1}^* \star C_{h_2}^* \star \cdots \star C_{h_n}^* \star D_{s_1}^+ \star D_{s_2}^- \star \cdots \star D_{s_{2m-1}}^+ \star D_{s_{2m}}^- (0).
\end{align*}
\end{lem}
\end{proof}

\subsection{Shape of rearrangement convolutions}
Lev's set rearrangement majorization Lemma \ref{lem:lev_set} is not only useful to prove the generalized rearrangement inequality, but also to show the shape of the rearrangement convolutions. One can easily show the following Lemma.
\begin{lem}\label{lem:shape}
If $f$ and $g$ are shape-compatible, and $a$ and $b$ are nonnegative real-values, all $f$ $g$, and $af+bg$ are again shape-compatible.
\end{lem}
\begin{prop}\label{prop:shape}
Assume $f$ is an arbitrary function. If probability mass functions $h$ is $\triangle$-regular, and $s_1$ and $s_2$ are $\square$-regular,
\begin{align*}
\mathfrak{I}^+ &\subseteq \mathfrak{I}\left( f^+ \star h^* \right),\\
\mathfrak{I}^- &\subseteq \mathfrak{I}\left( f^- \star h^* \right),\\
\mathfrak{I}^* &\subseteq \mathfrak{I}\left( s_1^+ \star s_2^- \right).
\end{align*}
\end{prop}
\begin{proof}
We show the first part $\mathfrak{I}\left( f^+ \star h^* \right) = \mathfrak{I}^+$ only. Other cases are similar. 

We consider the set-function decomposition used in the proof of Theorem~\ref{thm:decomposition} and Theorem \ref{thm:general_rearrangement_inequality}. Let $f=\sum a_i A_i$ and $h=\sum b_j B_j$ where $a_i=f^{[i]}-f^{[i+1]}$ and $b_j=f^{[j]}-f^{[j+1]}$ for $i=1,\cdots,p-1$ and $a_p=f^{[p]}$, $A_r=\{i_1,\cdots,i_r\}$, $b_p=g^{[p]}$, $B_r=\{j_1,\cdots,j_r\}$ for $\left( i_1,\cdots,i_p\right) \in \mathfrak{I}(f)$ and $\left( j_1,\cdots,j_p\right)\in \mathfrak{I}(h)$. As stated in the proof of Theorem \ref{thm:general_rearrangement_inequality}, since $h$ is $\triangle$-regular, $b_j=0$ when $j$ is even. So $h$ involves only $B_j$'s with odd sized sets. Then
\begin{align*}
f^+ \star h^* = \sum_{i, \text{odd } j} a_i b_j A_i^+ \star B_j^*.
\end{align*}
By Lemma \ref{lem:lev_set}, $\mathfrak{I}^+\subseteq \mathfrak{I}\left(A_i^+ \star B_j^*\right)$ for each $i$ and odd $j$. Then the proof follows by repeatedly applying Lemma \ref{lem:shape}.
\end{proof}

\subsection{All in one hand}
The rearrangement inequality of Theorem \ref{thm:general_rearrangement_inequality} implies the majorization with the help of one arbitrary function $g$ in Theorem \ref{thm:general_rearrangement_inequality}. A careful choice of $g$ enables us to derive the majorization from the generalized rearrangement inequality. 
\begin{thm}[Generalized rearrangement majorization]\label{thm:gen_maj}
Let $f$ be an arbitrary probability mass function. Assume $h_1,h_2,\cdots,h_n$ are $\triangle$-regular and $s_1,s_2,\cdots,$ $ s_{2m-1},s_{2m}$ are $\square$-regular,
\begin{align*}
f\star h_1\star h_2 \star &\cdots \star h_n \star s_1 \star s_2 \star \cdots\star  s_{2m-1} \star s_{2m} \\
& \prec f^{+/-} \star h_1^* \star h_2^* \star \cdots \star h_n^* \star s_1^+ \star s_2^- \star \cdots \star s_{2m-1}^+ \star s_{2m}^-
\end{align*}
where $f^{+/-}$ means that we can choose either $f^+$ or $f^-$.
\end{thm}
\begin{proof}
We show the case $f^+$ only in the right hand side. The other case can be proved in a similar way. We write $L:=f\star h_1\star h_2 \star \cdots \star h_n \star s_1 \star s_2 \star \cdots\star  s_{2m-1} \star s_{2m}$ and $R:=f^{+} \star h_1^* \star h_2^* \star \cdots \star h_n^* \star s_1^+ \star s_2^- \star \cdots \star s_{2m-1}^+ \star s_{2m}^-$ for abbreviation. 

We choose an ordered tuple $\left( i_1,i_2, \cdots, i_p\right)\in\mathfrak{I}\left( L \right)$. By choosing $g_r=\{-i_1,\cdots,-i_r\}$ for each $r=1,\cdots, p$, we have
\begin{align*}
\left( g_r \star L \right)(0) = \sum_{i=1}^{r} L^{[i]}.
\end{align*}
On the other hand, by repeatedly applying Proposition \ref{prop:shape}, we see that $\mathfrak{I}^+\subseteq \mathfrak{I}\left( R \right)$. Since $g_r^-$ is an indicator function on the first $r$ elements from $\mathfrak{I}^-$, we have
\begin{align*}
\left( g_r^- \star R \right) (0) = \sum_{i=1}^{r} R^{[i]}.
\end{align*}
Then Theorem \ref{thm:general_rearrangement_inequality} implies for each $r=1,\cdots, p$
\begin{align*}
\sum_{i=1}^{r} L^{[i]} \leq \sum_{i=1}^{r} R^{[i]}.
\end{align*}
The equality holds when $r=p$ since all are probability mass functions. Therefore the majorization follows.
\end{proof}
The following lemma shows why we have to control the shape in the summand of Theorem \ref{thm:main}.
\begin{lem}\label{lem:maj_sum}
If $f_1 \prec g_1$ and $f_2 \prec g_2$, and $g_1$ and $g_2$ are shape-compatible,
\begin{align*}
f_1 + f_2 \prec g_1 + g_2.
\end{align*}
\end{lem}
\begin{proof}
By Lemma \ref{lem:shape}, all $g_1$, $g_2$, and $g_1+g_2$ are shape-compatible. Then $\left( g_1+g_2\right)^{[r]}=g_1^{[r]}+g_2^{[r]}$ for each $r=1,\cdots, p$. In general, $\left( f_1+f_2\right)^{[r]}\leq f_1^{[r]}+f_2^{[r]}$. Therefore
\begin{align*}
\sum_{i=1}^{r}\left( f_1 + f_2 \right)^{[i]} \leq \sum_{i=1}^{r}\left( f_1^{[i]} + f_2^{[i]} \right)\leq \sum_{i=1}^{r}\left( g_1^{[i]} + g_2^{[i]} \right) = \sum_{i=1}^{r}\left( g_1 + g_2 \right)^{[i]}.
\end{align*}
When $r=p$, the equality follows easily by the majorization each.
\end{proof}

Now we are ready to prove our main Theorem \ref{thm:main}.
\begin{proof}[Proof of Theorem \ref{thm:main}]
We apply Theorem~\ref{thm:decomposition} for each $f_i$, so $f_i=f_{i,\triangle} + f_{i,\square}$. Then
\begin{align}\label{eqn:con_decomp}
f_1 \star f_2 \star \cdots \star f_n = \sum_{(\omega_1,\cdots,\omega_n)\in\{\triangle, \square\}^n} f_{1,\omega_1} \star f_{2,\omega_2} \star \cdots \star f_{n,\omega_n}.
\end{align}
Since each $w_i$ has two choices, the summands in the right hand side \eqref{eqn:con_decomp} can be $2^n$ choices. By letting $\omega=(\omega_1,\cdots,\omega_n)$, for each summand, we can find $\delta(i,\omega)$'s following the rule described in Section \ref{subsec:controlshape} below. Then by Theorem \ref{thm:gen_maj}, we have
\begin{align}\label{eq:decomposition}
f_{1,\omega_1} \star f_{2,\omega_2} \star \cdots \star f_{n,\omega_n} \prec f_{1,\omega_1}^{\delta(1,\omega)} \star f_{2,\omega_2}^{\delta(2,\omega)} \star \cdots \star f_{n,\omega_n}^{\delta(n,\omega)}
\end{align}
and $\mathfrak{I}^+\subseteq\mathfrak{I}\left( f_{1,\omega_1}^{\delta(1,\omega)} \star f_{2,\omega_2}^{\delta(2,\omega)} \star \cdots \star f_{n,\omega_n}^{\delta(n,\omega)} \right)$. Since all summands are $\mathfrak{I}^+$-compatible, Lemma \ref{lem:maj_sum} implies
\begin{align*}
f_1 \star f_2 \star \cdots \star f_n = &\sum_{(\omega_1,\cdots,\omega_n)\in\{\triangle, \square\}^n} f_{1,\omega_1} \star f_{2,\omega_2} \star \cdots \star f_{n,\omega_n}\\
 \prec &\sum_{(\omega_1,\cdots,\omega_n)\in\{\triangle, \square\}^n}f_{1,\omega_1}^{\delta(1,\omega)} \star f_{2,\omega_2}^{\delta(2,\omega)} \star \cdots \star f_{n,\omega_n}^{\delta(n,\omega)}.
\end{align*}
Then we choose a convex function $\Phi(x) = -x^\alpha$ for $\alpha\in (0,1)$, $\Phi(x)=x\log x$ for $\alpha=1$, and $\Phi(x)=x^\alpha$ for $\alpha\in (1,+\infty)$. Then the proof follows 
and decorating terms into R\'enyi entropy.
\end{proof}

\subsection{How to control each summand to be shape-compatible} \label{subsec:controlshape}
With the help of Proposition \ref{prop:shape}, we can control the shape of rearrangement convolutions. Similar to $\bar{\delta}(A,\delta'(A))$ in Section \ref{subsec:gen_rearrangement}, we define a function $\bar{\delta}(f,\delta'(f))$ by
\begin{displaymath}
\bar{\delta}(f,\delta'(f)) =  \left \{ 
\begin{array}{ll}
-1 & \text{if $f$ is $\square$-regular and $\delta'(f)=-$,}\\
0 & \text{if $f$ is $\triangle$-regular,}\\
+1 & \text{if $f$ is $\square$-regular and $\delta'(f)=+$.}
\end{array}
\right.
\end{displaymath}
Since each summand involves only $\triangle$-regular and/or $\square$-regular functions, the following Proposition is sufficient to control the shape.
\begin{prop}\label{prop:choose}
If $h_1,\cdots,h_n$ are $\triangle$-regular and $s_1,\cdots,s_m$ are $\square$-regular,
\begin{align*}
\mathfrak{I}^+\subseteq \mathfrak{I}\left( h_1^* \star \cdots \star h_n^* \star s_1^{\delta'(s_1)} \star \cdots \star s_m^{\delta'(s_m)} \right)
\end{align*}
where $\delta'(s_j)\in\{+,-\}$ if $\bar{\delta}:=\sum_{j=1}^{m} \bar{\delta}(s_j,\delta'(s_j)) \in\{0,+1\}$.
\end{prop}
\begin{proof}
For some $i$ and $j$, $\mathfrak{I}^*\subseteq\mathfrak{I}\left( s_i^{+} \star s_j^{-} \right) $ by Proposition \ref{prop:shape}. Then we see that $\mathfrak{I}^+\subseteq\mathfrak{I}\left(s_1^{\delta'(s_1)} \star \cdots \star s_m^{\delta'(s_m)} \right)$ if $\bar{\delta}:=\sum_{j=1}^{m} \bar{\delta}(s_j,\delta'(s_j)) \in\{0,+1\}$. \\
Since $\mathfrak{I}^*\subseteq\mathfrak{I}\left(h_1^* \star \cdots \star h_n^* \right)$, the proof follows by applying Proposition \ref{prop:shape} again.
\end{proof}
We remark that $\mathfrak{I}^-\subseteq\mathfrak{I}\left( h_1^* \star \cdots \star h_n^* \star s_1^{\delta'(s_1)} \star \cdots \star s_m^{\delta'(s_m)} \right)$ if $\bar{\delta} \in\{-1, 0\}$. Then there is no contraction to keep each summand to be $\mathfrak{I}^-$-compatible (instead of $\mathfrak{I}^+$) in the statement of Theorem \ref{thm:main}.

As a final remark, we would like to clarify the relationship between $\delta(i,\omega)$ and $\delta'(f)$. Although both functions define one of the shapes $+$ or $-$, $\delta(i,\omega)$ is specifically determined when $\delta'(f_{i,\omega_i})$'s are valid depending on the value of $\bar{\delta}=\sum\bar{\delta}(f_{i,\omega_i},\delta'(f_{i,\omega_i}))$. For example, for $n=2$, $\delta(1,\square,\square)=+$ is valid when $\delta'(f_{1,\square})=+$ and $\delta'(f_{2,\square})=-$ in Theorem \ref{thm:main} since $\bar{\delta}=0$. But $\delta(1,\square,\square)=+$ is not valid when $\delta'(f_{1,\square})=+$ and $\delta'(f_{2,\square})=+$ in Theorem \ref{thm:main} since $\bar{\delta}=+2$. To illustrate the validity, the following table shows all valid choices of $\delta(i,\omega)$'s when $n=2$.

\begin{center}
\begin{tabular}{ |c||c| } 
\hline
$\left(\delta(1,\omega), \delta(2,\omega)\right)$ & valid choices\\
\hline\hline
$\left(\delta({1,\triangle,\triangle}), \delta({2,\triangle,\triangle})\right)$ & $(+,+)$, $(+,-)$, $(-,+)$, $(-,-)$ \\ 
\hline
$\left(\delta({1,\triangle,\square}), \delta({2,\triangle,\square})\right)$ &
$(+,+)$, $(-,+)$ \\ 
\hline
$\left(\delta({1,\square,\triangle}), \delta({2,\square,\triangle})\right)$ &
$(+,+)$, $(+,-)$ \\ 
\hline
$\left(\delta({1,\square,\square}), \delta({2,\square,\square})\right)$ &
$(+,-)$, $(-,+)$ \\
\hline
\end{tabular}
\end{center}
Therefore total 32 combinations of $\delta(i,\omega)$'s are valid when $n=2$.

\section{Proof of Theorem~\ref{thm:same}}\label{sec:canon}



Focus on the right hand side of the equation $f_{1,\omega_1}^{\delta(1,\omega_1)} \star f_{2,\omega_2}^{\delta(2,\omega_2)} \star \cdots \star f_{n,\omega_n}^{\delta(n,\omega_n)}$ and first classify those terms into two $\triangle$ and $\square$ regular parts based on $\omega_i\in\{\triangle,\square\}$, respectively.
\begin{align*}
f_{1,\omega_1}^{\delta(1,\omega)} \star f_{2,\omega_2}^{\delta(2,\omega)} \star \cdots \star f_{n,\omega_n}^{\delta(n,\omega)} = \left( h_1^{\delta'(h_1)} \star \cdots \star h_k^{\delta'(h_k)} \right) \star \left( s_1^{\delta'(s_1)} \star \cdots \star s_m^{\delta'(s_m)} \right)
\end{align*}
where $h_1,\cdots,h_k$ are $\triangle$-regular, $s_1,\cdots,s_m$ are $\square$-regular among $f_{i,\omega_i}$'s, and $n=k+m$. For $\triangle$-regular $h_1,\cdots,h_k$, we know that $h_i^+=h_i^-$ for each $i$. Therefore, for any choice of $\delta_1'(\cdot)$ and $\delta_2'(\cdot)$
\begin{align}\label{eq:choicetriangle}
h_1^{\delta_1'(h_1)} \star \cdots \star h_k^{\delta_1'(h_k)} = h_1^{\delta_2'(h_1)} \star \cdots \star h_k^{\delta_2'(h_k)}.
\end{align}

For $\square$-regular $s_1,\cdots,s_m$, we prove the following lemma first.
\begin{lem}\label{lem:choicesquare}
If $s_1, \cdots, s_m$ are $\square$-regular,
\begin{align*}
s_1^{\delta'(s_j)} \star \cdots \star s_m^{\delta'(s_m)} \left(i \right) = s_1^{+} \star \cdots \star s_m^{+} \left( i -\left\lfloor \frac{m}{2}\right\rfloor \right) 
\end{align*}
where $\delta'(s_j)=+$ if $j$ is odd, and $\delta'(s_j)=-$ if $j$ is even. So $\bar{\delta}:=\sum_{j=1}^{m} \bar{\delta}(s_j,\delta'(s_j)) \in\{0,+1\}$.
\end{lem}
\begin{proof}
We prove this by induction. If $s_1$ is $\square$-regular, $\delta'(s_1)=+$. So
\begin{align*}
s_1^{\delta(s_1)}\left(i\right) = s_1^+\left(i-0\right).
\end{align*}
Now assume that Lemma \ref{lem:choicesquare} is true for $m\leq M$. First assume that $M$ is an odd number. Then by induction hypothesis,
\begin{align*}
s_1^{\delta(s_1)} \star \cdots \star s_M^{\delta(s_M)} \left(i \right) = s_1^{+} \star \cdots \star s_M^{+} \left( i - \frac{M-1}{2}\right) 
\end{align*}
Let $L_M :=s_1^{\delta'(s_1)} \star \cdots \star s_M^{\delta'(s_M)}$ and $R_M:= s_1^{+} \star \cdots \star s_M^{+}$. By Lemma \ref{lem:shape}, $L_M=L_M^+$. So for $s_{M+1}$, $\delta\left(s_{M+1}\right)=-$. Since $s_{M+1}^+(i)=s_{M+1}^-(i+1)$ by symmetry,
\begin{align*}
\left( L_M \star s_{M+1}^- \right) (i) = \left( L_M \star s_{M+1}^+ \right) \left( i - 1 \right) = \left( R_M \star s_{M+1}^+ \right)\left( i - \frac{M+1}{2}\right)
\end{align*}
where we used $L_M\left(i\right) = R_M\left(  i - \frac{M-1}{2} \right)$ to find the shifted amount in the second equality above.

If $M$ is an even number, $\delta(s_{M+1})=+$ and $L_M\left(i\right) = R_M\left(  i - \frac{M}{2} \right)$. Similar to the odd $M$ case,
\begin{align*}
\left( L_M \star s_{M+1}^+ \right) (i) = \left( R_M \star s_{M+1}^+ \right)\left( i - \frac{M}{2}\right).
\end{align*}
Hence the induction proof follows.
\end{proof}
By reordering $s_j$'s, we see that Lemma \ref{lem:choicesquare} is still true for any choice of $\delta(\cdot)$ satisfying $\bar{\delta}:=\sum_{j=1}^{m} \bar{\delta}(s_j,\delta(s_j)) \in\{0,+1\}$. i.e. we can drop the condition that $\delta(s_j)=+$ if $j$ is odd, and $\delta(s_j)=-$ if $j$ is even in Lemma \ref{lem:choicesquare}. Therefore Lemma \ref{lem:choicesquare} implies that for any different choices of $\delta_1(\cdot)$ and $\delta_2(\cdot)$
\begin{align}\label{eq:choicesquare}
s_1^{\delta_1'(s_1)} \star \cdots \star s_m^{\delta_1'(s_m)} (i) = s_1^{\delta_2'(s_1)} \star \cdots \star s_m^{\delta_2'(n_m)} (i) =  s_1^{+} \star \cdots \star s_m^{+} \left( i -\left\lfloor \frac{m}{2}\right\rfloor \right).
\end{align}

\section{Conclusion}

In this paper, we established a lower bound of the entropy of sums in prime cyclic groups and developed several applications
including to discrete entropy power inequalities, the Littlewood-Offord problem, and counting solutions for some linear equations. 
This required extending rearrangement inequalities in prime cyclic groups and combining them with function ordering 
in the sense of the majorization. 

In the sumset theory, the Cauchy-Davenport theorem is generalized by Kneser's theorem, which can be stated
for discrete abelian groups (see, e.g., \cite{TV06:book, Gry10, Gry13:book}). 
An entropy version of Kneser's theorem would be very interesting but is currently unavailable.
It would also be interesting to generalize our results to all cyclic groups $\mathbb{Z}\big/{n\mathbb{Z}}$ for all $n\in\mathbb{N}$,
and in another direction to the multidimensional case $(\zpz)^d$ for $d>1$.

\section*{Acknowledgments}
This work was supported in part by the U.S. National Science Foundation through grants DMS-1409504 (CAREER) and CCF-1346564.


\end{document}